\documentclass[a4paper, 11pt]{amsart}
\usepackage[T1]{fontenc}
\usepackage[utf8]{inputenc}
\usepackage{amssymb, amsmath, amsthm, enumerate, color}
\usepackage[shortlabels]{enumitem}
\usepackage[usenames,dvipsnames,svgnames,table]{xcolor}
\usepackage[foot]{amsaddr}

\usepackage[margin=28mm]{geometry}
\renewcommand{\baselinestretch}{1.09}
\allowdisplaybreaks

\usepackage[unicode=true]{hyperref}
\hypersetup{
    citecolor={blue!50!black},
    pdftitle=Edge separators for graphs excluding a minor
}
\usepackage{cleveref}

\DeclareTextFontCommand{\defn}{\em \color{Maroon}}

\newtheorem{theorem}{Theorem}
\newtheorem{corollary}[theorem]{Corollary}
\newtheorem{lemma}[theorem]{Lemma}
\newtheorem{fact}[theorem]{Fact}

\theoremstyle{definition}

\newcommand{\tw}{\operatorname{tw}}

\let\le\leqslant
\let\ge\geqslant

\title{Edge separators for graphs excluding a minor}

\author[G.~Joret]{Gwena\"el Joret}
\address[G.~Joret]{Computer Science Department \\
  Universit\'e Libre de Bruxelles\\
  Brussels, Belgium}
\email{gwenael.joret@ulb.be}

\author[W.~Lochet]{William Lochet}
\address[W.~Lochet]{LIRMM, Université de Montpellier, CNRS, Montpellier, France,} 
\email{william.lochet@gmail.com}

\author[M.~T.~Seweryn]{Michał T.\ Seweryn}
\address[M.~T.~Seweryn]{Computer Science Department \\
  Universit\'e Libre de Bruxelles\\
  Brussels, Belgium}
\email{michal.seweryn@ulb.be}

\thanks{G.\ Joret and M.T.\ Seweryn are supported by a PDR grant from the Belgian National Fund for Scientific Research (FNRS). G.\ Joret is also supported by a CDR grant from FNRS, and by the Wallonia Brussels International (WBI) agency.}

\begin{document}

\maketitle

\begin{abstract}
We prove that every $n$-vertex $K_t$-minor-free graph $G$ of maximum degree $\Delta$ has a set $F$ of $O(t^2(\log t)^{1/4}\sqrt{\Delta n})$ edges such that every component of $G - F$ has at most $n/2$ vertices. This is best possible up to the dependency on $t$ and extends earlier results of Diks, Djidjev, S{\`y}kora, and Vr\v{t}o (1993) for planar graphs, and of S{\`y}kora and Vr\v{t}o (1993) for bounded-genus graphs. Our result is a consequence of the following more general result: The line graph of $G$ is isomorphic to a subgraph of the strong product $H \boxtimes K_{\lfloor p \rfloor}$ for some graph $H$ with treewidth at most $t-2$ and $p = \sqrt{(t-3)\Delta |E(G)|} + \Delta$.
\end{abstract}

\section{Introduction}

A \defn{balanced vertex separator} of an \(n\)-vertex graph \(G\)
is a set \(X \subseteq V(G)\) such that every component of
\(G - X\) has at most \(n/2\) vertices.\footnote{Note that it implies that the components of $G-X$ can be grouped into two sets each with at most $2n/3$ vertices; this is sometimes used as the definition of `balanced vertex separator' in the literature.} 
The well-known Planar Separator Theorem by Lipton and Tarjan%
~\cite{lipton1979separator}
states that every \(n\)-vertex planar graph has
a balanced vertex separator of
size \(O(\sqrt{n})\). 
Alon, Seymour and Thomas~\cite{alon1990separator} showed that
every \(n\)-vertex \(K_t\)-minor-free graph
has a balanced vertex separator of size at most \(t^{3/2}\sqrt{n} \). 

In this paper, we study balanced edge separators.
A \defn{balanced edge separator} of an \(n\)-vertex graph \(G\)
is a set \(F \subseteq E(G)\) such that every component of
\(G - F\) has at most  \(n/2\) vertices. 
The aforementioned classes
of graphs with balanced vertex separators of size
\(O(\sqrt{n})\)
do not admit balanced edge separators of size \(o(n)\);
indeed, the smallest balanced edge separator of the
\(n\)-vertex star \(K_{1, n-1}\) has size \(\lceil n/2 \rceil\).

The star, however, has a vertex of degree \(n-1\).
If we assume that the maximum degree \(\Delta\) of \(G\) is
sublinear in \(n\), then in some cases
we can retrieve sublinear balanced edge separators.
Diks, Djidjev, S{\`y}kora, and Vrťo~\cite{diks1993edge} showed that if \(G\) is planar,
then $G$ has a balanced edge separator of size \(O(\sqrt{\Delta n})\),
and S{\`y}kora and Vrťo~\cite{sykora1993edge} showed that
if the Euler genus of \(G\) is \(g\), then there exists a balanced edge
separator of size \(O(\sqrt{g\Delta n})\). These results are true also
in the weighted setting where each vertex \(x \in V(G)\) is assigned
a weight \(w(x)\) with \(0 \le w(x) \le \frac{1}{2}\), the total weight
of all vertices is \(1\), and the edge separator should split
\(G\) into components of weight at most \(\frac{1}{2}\).

Laso{\'n} and Sulkowska~\cite{lason2021modularity} showed that
in the weighted setting, if \(G\) is an \(n\)-vertex \(K_t\)-minor-free
graph of maximum degree \(\Delta = o(n)\) and the vertices are
weighted proportionally to their degrees, then there exists a
balanced edge separator of size \(o(n)\). 
Their proof relies on spectral methods---more precisely, on an upper bound for the second smallest eigenvalue of the Laplacian matrix of \(K_t\)-minor-free
graphs due to Biswal, Lee, and Rao~\cite{BLR10}---and only works for these specific weights.  
They asked if one can always find a balanced edge separator of size \(O_t(\sqrt{\Delta n})\) for any weights (including uniform weights), as is the case for planar graphs and graphs of bounded genus. 
In this paper, we give an affirmative answer to this question.

\begin{theorem}\label{thm:edge-separator}
  Let \(t \ge 3\), Let \(G\) be an \(n\)-vertex
  \(K_t\)-minor-free graph of maximum degree \(\Delta\),
  and let \(w \colon V(G) \to [0, \frac{1}{2}]\) be a weight function
  such that \(\sum_{x \in V(G)} w(x) = 1\). Then there exists a set
  \(F \subseteq E(G)\) with
  \[
    |F| \le (t-1)\left\lfloor\sqrt{(t-3)|E(G)|\Delta} + \Delta\right\rfloor  
  \in O\left(t^{2} (\log t)^{1/4}\cdot \sqrt{\Delta n}\right)
  \]
  such that
  \(\sum_{x \in V(C)}w(x) \le \frac{1}{2}\) for each component \(C\) of
  \(G - F\). 
\end{theorem}

This result is best possible up to dependency on \(t\): S{\`y}kora and Vrťo~\cite{sykora1993edge} showed that there exist $n$-vertex planar graphs \(G\) of maximum degree $\Delta$ such that every balanced edge separator has size \(\Omega(\sqrt{\Delta n})\). 

We actually prove the following stronger result. 

\begin{theorem}\label{thm:line-graph-product-structure}
  Let \(t \ge 3\) and let \(G\) be a \(K_t\)-minor-free graph of
  maximum degree \(\Delta\) with
  \(m\) edges. Then the line graph of \(G\) is isomorphic to
  a subgraph of the strong product \(H \boxtimes K_{\lfloor p \rfloor}\)
  for some graph \(H\) with
  \(\tw(H) \le t - 2\) and \(p = \sqrt{(t-3)\Delta m} +\Delta\). 
\end{theorem}

The \defn{strong product \(H \boxtimes K\)} of graphs \(H\) and \(K\) is a graph on \(V(H) \times V(K)\) where \((x_1, y_1)\) and
\((x_2, y_2)\) are adjacent if \(x_1 = x_2\) and \(y_1 y_2 \in E(K)\), or
\(x_1 x_2 \in E(H)\) and \(y_1 = y_2\), or
\(x_1 x_2 \in E(H)\) and \(y_1 y_2 \in E(K)\). 
(When $K$ is a complete graph on $p$ vertices, taking the strong product of $H$ with $K$ amounts to `blowing up' each vertex of $H$ by a clique of size $p$.) 
Theorem~\ref{thm:line-graph-product-structure} directly implies the following upper bound on the treewidth of the line graph of \(G\). 

\begin{theorem}\label{thm:line-graph-tw}
  Let \(t \ge 3\), and let \(G\) be a \(K_t\)-minor free graph of
  maximum degree \(\Delta\) with \(m\) edges. Then the line
  graph of \(G\) has treewidth at most \((t-1) \lfloor \sqrt{(t-3)\Delta m} + \Delta\rfloor - 1\).
\end{theorem}

Theorem~\ref{thm:edge-separator} then follows  from Theorem~\ref{thm:line-graph-tw} by a simple argument on the tree-decomposition provided by the latter theorem (see Section~\ref{sec:proofs}). 

Theorem~\ref{thm:line-graph-product-structure} can be thought of as an `edge version' of the following recent strengthening of the balanced vertex separator result by Alon, Seymour and Thomas~\cite{alon1990separator} due to Illingworth, Scott and Wood~\cite{illingworth2021alon}: Every \(n\)-vertex \(K_t\)-minor free is isomorphic to a subgraph of the strong product
\(H \boxtimes K_{\lfloor p \rfloor}\) where \(\tw(H) \le t-2\) and \(p = \sqrt{(t-3)n}\). 
The authors of~\cite{illingworth2021alon} established their result by modifying the proof in~\cite{alon1990separator}. 
Our proof of Theorem~\ref{thm:line-graph-product-structure} is likewise a modification of the proof in~\cite{alon1990separator} and relies heavily on the insights from~\cite{illingworth2021alon}; the main work consists in adapting to the edge setting.

One consequence of Theorem~\ref{thm:edge-separator} is an upper bound for the isoperimetric number (a.k.a.\ edge expansion or Cheeger constant) of \(K_t\)-minor-free graphs. 
The \defn{isoperimetric number \(\phi(G)\)} of a
graph \(G\) is defined as
\[
  \phi(G) = \min \left\{\frac{|\{x y \in E(G): x \in S, y \not \in S\}|}{|S|}:
  S \subseteq V(G), 1\le |S| \le |V(G)|/2\right\}.
\]

\begin{corollary}\label{cor:isoperimetric_number}
For \(t \ge 3\), every \(n\)-vertex \(K_t\)-minor-free graph \(G\) with maximum degree $\Delta$ satisfies
\[
  \textstyle \phi(G) = O\left(t^2 (\log t)^{1/4} \cdot \sqrt{\frac{\Delta}{n}}\right)
\]
\end{corollary}
\begin{proof}
  Let \(F\) be a balanced edge separator of \(G\) with
  \(|F| = O(t^2 (\log t)^{1/4} \cdot \sqrt{\Delta n})\) given by Theorem~\ref{thm:edge-separator}. 
  Since each component of \(G - F\) has at most \(n/2\) vertices,
  we may choose a subset of these components so that the union \(S\) of their
  vertex sets satisfies \(\frac{1}{3}n \le |S| \le \frac{1}{2}n\). 
  It follows 
  \[
  \phi(G) \le \frac{|\{x y \in E(G) : x \in S, y \not \in S\}|}{|S|} \le \frac{|F|}{n/3}
  = O\left(t^2 (\log t)^{1/4} \cdot \sqrt{\frac{\Delta}{n}}\right).\qedhere
  \] 
\end{proof}

The bound in Corollary~\ref{cor:isoperimetric_number} is best possible up to the dependence on $t$, and extends previous bounds for planar graphs~\cite{diks1993edge} and bounded-genus graphs~\cite{sykora1993edge}. 

Our proofs are constructive, and in particular, there exists a polynomial time algorithm, which given a graph \(G\) and an integer \(t\), outputs a set \(F\) as in Theorem~\ref{thm:edge-separator} and a graph \(H\) as in Theorem~\ref{thm:line-graph-product-structure}
together with an isomorphism between
the line graph of \(G\) and a subgraph of
\(H \boxtimes K_{\lfloor p \rfloor}\) where
\(p = \sqrt{(t-3)\Delta m} + \Delta\).

In Section~\ref{sec:preliminaries} we introduce all the necessary
definitions, notations and preliminary results, and in
Section~\ref{sec:proofs} we prove Theorems~\ref{thm:edge-separator}, \ref{thm:line-graph-product-structure} and~\ref{thm:line-graph-tw}.

\section{Preliminaries}\label{sec:preliminaries}
We consider simple finite undirected graphs \(G\)
with vertex set \(V(G)\)
and edge set \(E(G)\). For subsets \(X, Y \subseteq V(G)\), we denote
by \defn{\(E_G(X, Y)\)} the set of all edges \(x y \in E(G)\) with \(x \in X\)
and \(y \in Y\).
We denote by \defn{\(N_G(X)\)} the open neighborhood of a set
\(X \subseteq V(G)\), i.e.\ the set of all vertices outside $X$ that are adjacent to at least one vertex from \(X\). 
We drop the subscripts from the notations \(E_G(X, Y)\) and \(N_G(X)\) when the graph $G$ is clear from the context. 

A set of vertices \(U \subseteq V(G)\) is \defn{connected} in a graph \(G\) if the induced subgraph \(G[U]\) is connected. For \(t \ge 1\), a \defn{\(K_t\)-model} in 
\(G\) is a sequence \((U_1, \ldots, U_t)\) of pairwise disjoint connected subsets of \(V(G)\) such that
\(E(U_i, U_j) \neq \emptyset\) for distinct \(i, j \in \{1, \ldots, t\}\).
Note that $G$ contains \(K_t\) as a minor if and only if $G$ has a \(K_t\)-model.

A \defn{tree-decomposition} 
of a graph \(G\) is a family \((B_u)_{u \in V(T)}\)
of subsets of \(V(G)\) called \defn{bags}, indexed by nodes of a tree \(T\), such that
\begin{itemize}
  \item \(V(G) = \bigcup_{u\in V(T)} B_u\),
  \item for every \(x y \in E(G)\) there exists
    \(u \in V(T)\) with \(\{x, y\} \subseteq B_u\), and
  \item For all \(u_1, u_2, u_3 \in V(T)\) such that \(u_2\) lies on the path between \(u_1\)
  and \(u_3\) in \(T\), we have \(B_{u_1} \cap B_{u_3} \subseteq B_{u_2}\).
\end{itemize}
The \defn{width} of \((B_u)_{u \in V(T)}\) is
\(\max\{|B_u| - 1 : u \in V(T)\}\).
The \defn{treewidth} of a graph \(G\), denoted \defn{\(\tw(G)\)},
is the minimum width of its tree-decompositions.
The following fact summarizes some simple properties of
tree-decompositions that we use in our proof.

\begin{fact}\label{fact}
  Given a graph \(G\) with a tree-decomposition \((B_u)_{u \in V(T)}\) of width at most
  \(k\), the following properties hold:
  \begin{enumerate}[(\roman*)]
    \item\label{fact:complete-intersection} For every graph \(G'\) such that
    \(\tw(G') \le k\) and \(G \cap G'\) is a (possibly empty) complete graph,
    we have \(\tw(G \cup G') \le k\).
    \item\label{fact:attaching-to-clique} Every graph \(G'\) obtained from
    \(G\) by adding a new vertex adjacent to a clique of $G$ of size at most \(k\) has treewidth at most \(k\).
    \item\label{fact:clique} For every clique \(X\) in \(G\) there exists \(u \in V(T)\) with \(X \subseteq B_u\).
    \item\label{fact:connected} For every connected set \(U \subseteq V(G)\),
    the set \(\{u \in V(T): B_u \cap U \neq \emptyset\}\) is connected in \(T\).
  \end{enumerate}
\end{fact}

The \defn{line graph \(L(G)\)} of a graph \(G\) is a graph whose
vertex set is \(E(G)\) and in which distinct edges
\(e, e' \in E(G)\) are adjacent if they share a common end in \(G\).

Given a graph $G$, a \defn{graph-partition} of \(G\) is a graph \(H\) such that the vertex set of \(H\) is a partition of \(V(G)\) into nonempty parts, and for all distinct \(X, Y \in V(H)\) we have \(X Y \in E(H)\) if \(E_G(X, Y) \neq \emptyset\).  
(Note that it is allowed to have \(X Y \in E(H)\) in case \(E_G(X, Y) = \emptyset\), there is no restriction in this case.) 
For an integer \(k\) and a real \(p\), we call 
a graph-partition \(H\) of \(G\) a \defn{\((k, p)\)-partition}
if  \(\tw(H) \le k\) and \(|X| \le p\) for each \(X \in V(H)\).
Observe that if a graph \(G\) has a \((k, p)\)-partition then $G$ is isomorphic to a subgraph of \(H \boxtimes K_{\lfloor p \rfloor}\) for some graph \(H\) with \(\tw(H) \le k\). 
In the proofs, we will often consider a graph-partition \(H\) of a graph $G$ together with some distinguished clique \(\{X_1, \ldots, X_h\}\) of $H$; in this case, we say that \(H\) is \defn{rooted} at \(\{X_1, \ldots, X_h\}\).

\section{The proofs}\label{sec:proofs}

We need the following lemma by Alon, Seymour and Thomas~\cite{alon1990separator}.

\begin{lemma}[\cite{alon1990separator}]\label{lem:ast}
  Let \(G\) be a graph, let \(A_1, \ldots, A_h\) be \(h\) subsets of \(V(G)\), and let $r$ be a real number with \(r \ge 1\).
  Then either: 
  \begin{itemize}
    \item there is a tree \(T\) in \(G\) with \(|V(T)| \le r\) such that 
    \(V(T) \cap A_i \neq \emptyset\) for every \(i \in \{1, \ldots, h\}\), or
    \item there exists \(Z \subseteq V(G)\) with \(|Z| \le (h-1)|V(G)|/r\)
    such that no component of \(G - Z\) intersects
    all of \(A_1, \ldots, A_h\).
  \end{itemize}
\end{lemma}

We deduce an edge variant of this lemma.

\begin{lemma}\label{lem:line-ast}
  Let \(G\) be a graph without isolated vertices, let
  \(A_1, \ldots, A_h\) be \(h\) subsets of \(V(G)\), and let $r$ be a real number with \(r \ge 1\). 
  Then either: 
  \begin{itemize}
    \item\label{itm:tree} there is a tree \(T\) in \(G\) with \(|E(T)| \le r\) such that 
    \(V(T) \cap A_i \neq \emptyset\) for every \(i \in \{1, \ldots, h\}\), or
    \item\label{itm:separator} there exists \(F \subseteq E(G)\) with \(|F| \le (h-1)|E(G)|/r\)
    such that no component of \(G - F\) intersects
    all of \(A_1, \ldots, A_h\).
  \end{itemize}
\end{lemma}
\begin{proof}
  Apply Lemma~\ref{lem:ast} to \(L(G)\), the sets
  \(E_G(A_1, V(G)), \ldots, E_G(A_h, V(G))\) and \(r\). 
  If \(L(G)\) contains a tree \(T_0\) with \(|V(T_0)| \le r\) such that 
  \(V(T_0) \cap E_G(A_i, V(G)) \neq \emptyset\) for every \(i \in \{1, \ldots, h\}\),
  then \(G\) contains a tree \(T\) such that \(E(T) = V(T_0)\), and thus
  \(|E(T)| = |V(T_0)| \le r\)
  and \(V(T) \cap A_i \neq \emptyset\) for every \(i \in \{1, \ldots, h\}\). 
  
  Now suppose that \(L(G)\) contains a set
  \(Z \subseteq V(L(G))\) with
  \(|Z| \le (h-1)|V(G)|/r\)
  such that no component of \(L(G) - Z\)
  intersects
  each \(E_G(A_i, V(G))\) with
  \(i \in \{1, \ldots, h\}\).
  Let \(F := Z\). 
  If no component of \(G - F\) intersects all of \(A_1, \ldots, A_h\), then $F$ satisfies the lemma, and we are done. 
  Assume thus  that some component \(C\) of \(G - F\) 
  intersects all of \(A_1\), \ldots,
  \(A_h\). By our assumption on \(F\), 
  there exists \(i \in \{1, \ldots, h\}\)
  such that \(E(C) \cap E_G(A_i, V(G)) = \emptyset\). Hence, \(C\) must consist of a single vertex that belongs to all of
  \(A_1\), \ldots, \(A_h\), and therefore the tree \(T := C\) satisfies the lemma. 
\end{proof}

The following lemma is the heart of the proofs of our results.

\begin{lemma}\label{lem:line-induction}
   Let $t, \Delta, m, h$ be integers with \(t \ge 3\) and \(\Delta, m, h \ge 1\), and let \(p := \sqrt{(t-3)\Delta m}+\Delta\).
   Let \(G\) be a connected \(K_t\)-minor-free graph of
   maximum degree at most \(\Delta\) with \(m\) edges,
   let \(C\) be a proper induced subgraph of \(G\) with \(|V(C)| \ge 1\), 
   and let  \(E_1, \ldots, E_h\) be disjoint nonempty subsets of
   \(E(G) \setminus E(C)\) such that
   \(|E_i| \le p\) for each 
   \(i \in \{1, \ldots, h\}\).
   If there exists a \(K_h\)-model \((U_1, \ldots, U_h)\) in
   \(G - V(C)\) such that
   \(N(V(C)) \subseteq U_1 \cup \cdots \cup U_h\) and
   \(E(V(C), U_i) \subseteq E_i\) for each \(i \in \{1, \ldots, h\}\),
   then \(L(G)[E(C) \cup E_1 \cup \cdots \cup E_h]\) admits a \((t-2,p)\)-partition \(H\) rooted at
   \(\{E_1, \ldots, E_h\}\).
\end{lemma}
\begin{proof}
  We prove the lemma by induction on the value \(2|V(C)| + h\).
  Since \(G\) is \(K_t\)-minor free, we have \(h \le t - 1\).
  
  Suppose that \(C\) is disconnected, say \(C\) is the union of vertex disjoint graphs \(C_1\) and \(C_2\) with \(|V(C_{\alpha})|\geq 1\) for \(\alpha \in \{1, 2\}\). 
  For each \(\alpha \in \{1, 2\}\), we have \(2|V(C_\alpha)| + h < 2 |V(C)| + h\),
  \(N(V(C_\alpha)) \subseteq N(V(C))\) and
  \(E(V(C_\alpha), U_i) \subseteq E(V(C), U_i)\) for each \(i \in \{1, \ldots, h\}\), so we may apply
  the induction hypothesis to \(C_\alpha\) to obtain a \((t-2,p)\)-partition \(H_\alpha\) of
  \(L(G)[E(C_\alpha) \cup E_1 \cup \cdots \cup E_h]\) rooted at
   \(\{E_1, \ldots, E_h\}\).
  By Fact~\ref{fact}\ref{fact:complete-intersection}, \(H = H_1 \cup H_2\) is the
  desired \((t-2, p)\)-partition of \(L(G)[E(V(C) \cup E_1 \cup \cdots \cup E_h)]\).
  Therefore, we may assume that \(C\) is connected.

  For each \(i \in \{1, \ldots, h\}\), let \(A_i = V(C) \cap N(U_i)\).
  Suppose that some \(A_i\) is empty, say without loss of generality \(A_h = \emptyset\).
  Since \(G\) is connected, not all sets \(A_i\) are empty, so
  \(h \ge 2\). 
  We have \(2|V(C)| + (h-1) < 2|V(C)| + h\) and
  \(N(V(C)) \subseteq U_1 \cup \cdots \cup U_{h-1}\)
  because \(N(V(C)) \cap U_h = \emptyset\),
  so we may apply the induction hypothesis to
  \(C\) and the sets \(E_1, \ldots, E_{h-1}\) to obtain a  \((t-2, p)\)-partition \(H_0\) of \(L(G)[E(C) \cup E_1 \cup \cdots \cup E_{h-1}]\)
  rooted at \(\{E_1, \ldots, E_{h-1}\}\).
  Since \(A_h = \emptyset\), no edge of \(E_h\) is incident to an edge in \(E(C)\), and thus the line graph \(L(G)\) does not contain any edges between \(E_h\) and \(E(C)\). 
  Hence, by Fact~\ref{fact}\ref{fact:attaching-to-clique}, the desired \((t-2, p)\)-partition \(H\) of \(L(G)[E(C) \cup E_1 \cup \cdots \cup E_h]\) can be obtained from \(H_0\) by adding \(E_h\) as a new vertex adjacent to \(E_1, \ldots, E_{h-1}\). 
  Therefore, we may assume that the sets \(A_1, \ldots, A_h\) are nonempty.
  
  Since \((V(C), U_1, \ldots, U_h)\) is a \(K_{h+1}\)-model in \(G\)
  and \(G\) is \(K_t\)-minor-free, we have \(h \le t - 2\).
  If \(|V(C)| = 1\), then \(E(C) = \{\emptyset\}\), and the complete graph on \(\{E_1, \ldots, E_h\}\) gives the desired
  \((t-2, p)\)-partition \(H\), so we may assume that \(|V(C)| > 1\).

  Suppose that \(C\) contains a tree \(T\) on  at most \(\sqrt{(h-1)m/\Delta}+1\) vertices that contains at least one vertex in 
  \(A_i\) for each \(i \in \{1, \ldots, h\}\). Let \(U_{h+1} = V(T)\), so that \((U_1, \ldots, U_{h+1})\) is a \(K_{h+1}\)-model, and let \(E_{h+1} := E(V(T)) \cup E(V(T), V(C)-V(T))\). 
  Note that $E_{h+1}$ is nonempty, since \(C\) is connected and \(|V(C)| > 1\). 
  Observe that
  \[
    |E_{h+1}| \le \Delta\cdot|V(T)| \le \Delta\cdot\left(\sqrt{(h-1)m/\Delta} + 1\right) \le \sqrt{(t-3)\Delta m} + \Delta = p.
  \]  
  If $V(T)=V(C)$ then the complete graph on \(\{E_1, \ldots, E_{h+1}\}\) gives the desired \((t-2, p)\)-partition \(H\) of  \(L(G)[E(C) \cup E_1 \cup \cdots \cup E_h]\). (Recall that \(h \le t - 2\).) 

  If $V(T)\neq V(C)$, then we have \(2\cdot|C - V(T)| + (h + 1) \le 2(|V(C)| - 1) + (h + 1) < 2 |V(C)| + h\), so we may apply the induction hypothesis
  to \(C - V(T)\),  \((U_1, \ldots, U_{h+1})\) and the sets \(E_1\), \ldots, \(E_{h+1}\) to obtain a \((t-2,p)\)-partition \(H\) of \(L(G)[E(C-V(T)) \cup E_1 \cup \cdots \cup E_{h+1}]\) rooted at \(\{E_1, \ldots, E_{h+1}\}\). 
  Since \(E(C) = E(C - V(T)) \cup E_{h+1}\), the partition $H$ is a \((t-2,p)\)-partition of  \(L(G)[E(C) \cup E_1 \cup \cdots \cup E_h]\) rooted at
   \(\{E_1, \ldots, E_h\}\), as desired.

  It remains to consider the case when no tree in \(C\) with at most
  \(\sqrt{(h-1)m/\Delta} + 1\) vertices intersects all of
  \(A_1, \ldots, A_h\). In particular, \(h \ge 2\). 
  Since \(C\) is connected and \(|V(C)| > 1\), $C$ does not contain isolated vertices. 
  By Lemma~\ref{lem:line-ast} with \(r=\sqrt{(h-1)m/\Delta}\), there exists a set \(F \subseteq E(C)\) with
  \[
    |F| \le (h-1)m/\sqrt{(h-1)m/\Delta} = \sqrt{(h-1)\Delta m}
    \le \sqrt{(t-3)\Delta m} < p
  \]
  such that no component of \(C - F\) intersects all of 
  \(A_1, \ldots, A_h\).
  Among all such sets \(F\), choose a smallest one.
  Since \(C\) is connected and the sets \(A_1, \ldots, A_h\) are
  nonempty, \(F \neq \emptyset\).

  Let \(C_1, \ldots, C_s\) be the components of \(C - F\). 
  Observe that \(C_1, \ldots, C_s\) are induced subgraphs of $C$ (and thus of $G$), by our choice of $F$. 
  Our goal now is to show that for each \(j \in \{1, \ldots, s\}\),
  there exists a \((t-2,p)\)-partition \(H_j\) of \(L(G)[E(C_j) \cup F \cup E_1 \cup \cdots \cup E_h]\)
  rooted at \(\{F, E_1, \ldots, E_h\}\). 
  By Fact~\ref{fact}\ref{fact:complete-intersection}, this will then imply that \(L(G)[E(C) \cup E_1 \cup \cdots \cup E_h]\) admits a \((t-2,p)\)-partition \(H\) rooted at \(\{E_1, \ldots, E_h\}\), as desired. 

  Towards this goal, fix \(j \in \{1, \ldots, s\}\), and let \(i' \in \{1, \ldots, h\}\) be such that
  \(V(C_j) \cap A_{i'} = \emptyset\).
  Consider the sets \(E_1', \ldots, E_h'\) where \(E_{i}' = E_i\)
  for \(i \neq i'\) and \(E_{i'}' = F\).
  
  Let \(X\) denote the set of all vertices of \(C\) that lie in some component of \(C - F\) containing at least one
  vertex from \(A_{i'}\).
  Since each component of \(C[X]\) contains a vertex from \(A_{i'}\), the set
  \(U_{i'} \cup X\) is connected in \(G\).
  Let \((U_1', \ldots, U_h')\) be the \(K_h\)-model where
  \(U_i' = U_i\) for \(i \neq i'\) and \(U_{i'}' = U_{i'} \cup X\).
  By the minimality of \(F\),
  each edge \(e \in F\) with an end in \(C_j\) belongs to one component
  of \(C - (F \setminus \{e\})\) with an element of \(A_i'\),
  so \(e\) has an end in \(U_{i'}'\). Hence,
  \(N(C_j) \subseteq (U_1' \cup \cdots \cup U_{h}')\). 
  Note also that \(2|V(C_j)| + h < 2|V(C)| + h\). 
  Therefore, we may apply the induction hypothesis to \(C_j\) and
  the sets \(E_1', \ldots, E_h'\) to obtain a \((t-2, p)\)-partition
  \(H_j^0\) of \(L(G)[E(C_j) \cup E_1' \cup \cdots \cup E_h']\)
  that is rooted at \(\{E_1', \ldots, E_h'\}\).
  
  Since \(V(C_j) \cap A_{i'} = \emptyset\), and since \(E_{i'} \subseteq E(G) \setminus E(C)\), no edge of $E_{i'}$ is incident to an edge in \(E(C_j)\). 
  Thus, the line graph
  \(L(G)\) does not contain any edges between \(E_{i'}\) and \(E(C_j)\). 
  Since \(E(V(C), U_{i'}) \subseteq E_{i'}\) and \(h \le t - 2\), we can
  obtain a \((t-2, p)\)-partition \(H_j\) of
  \(L(G)[E(C_j) \cup F \cup E_1 \cup \cdots \cup E_h]\) rooted at
  \(\{F, E_1, \ldots, E_h\}\) from \(H_j^0\) 
  by adding \(E_{i'}\) as a new vertex adjacent to \(E_1', \ldots, E_h'\).
  
  By Fact~\ref{fact}\ref{fact:complete-intersection},
  \(H := H_1 \cup \cdots \cup H_s\) is a \((t-2, p)\)-partition
  of \(L(G)[E(C) \cup E_1 \cup \cdots \cup E_h]\) rooted at
  \(\{F, E_1, \ldots, E_h\}\) and in particular at
  \(\{E_1, \ldots, E_h\}\). 
  This concludes the proof of the lemma. 
\end{proof} 

\begin{proof}[Proof of Theorem~\ref{thm:line-graph-product-structure}]
  Let \(G\) be a \(K_t\)-minor-free graph 
  of maximum degree \(\Delta\) with \(m\) edges, and let
  \(G_1\), \ldots, \(G_s\) be the components of \(G\).
  For each \(j \in \{1, \ldots, s\}\), we construct a \((t-2, p)\)-partition
  \(H_j\) of \(L(G_j)\).
  If \(G_j\) is an isolated vertex, then \(L(G_j)\) is
  an empty graph and we can take the empty graph as \(H_j\).
  If \(G_j\) is not an isolated vertex, then choose any vertex \(x \in V(G_j)\).
  By Lemma~\ref{lem:line-induction} applied to \(C = G_j - x\),
  \(E_1 = E(\{x\}, V(G))\) and \(U_1 = \{x\}\),
  \(L(G_j)\) has a \((t-2, p)\)-partition \(H_j\).
  Hence, by Fact~\ref{fact}\ref{fact:complete-intersection},
  \(H = H_1 \cup \cdots \cup H_s\) is a \((t-2, p)\)-partition of \(L(G)\),
  and therefore \(L(G)\) is isomorphic to a subgraph of \(H \boxtimes K_{\lfloor p \rfloor}\).
\end{proof}

\begin{proof}[Proof of Theorem~\ref{thm:line-graph-tw}]
  Let \(G\) be a \(K_t\)-minor-free graph 
  of maximum degree \(\Delta\) with \(m\) edges.
  By Theorem~\ref{thm:line-graph-product-structure}, there exists a graph
  \(H\) with \(\tw(H) \le t-2\) such that
  \(L(G)\) is isomorphic to a subgraph of \(H \boxtimes K_{\lfloor p \rfloor}\)
  where \(p = \sqrt{(t-3)\Delta m} +\Delta\).
  If \((B_u)_{u \in V(T)}\) is a tree-decomposition of \(H\) of width
  at most \(t - 2\), then \((B_u \times V(K_{\lfloor p \rfloor}))_{u \in V(T)}\)
  is a tree-decomposition of \(H \boxtimes K_{\lfloor p \rfloor}\) of width at most
  \((t-1)\lfloor p\rfloor - 1\), so
  \[\tw(L(G)) \le \tw(H \boxtimes K_{\lfloor p \rfloor})
  \le (t-1)\lfloor p \rfloor - 1 = (t-1)\left\lfloor \sqrt{(t-3)\Delta m}+\Delta\right\rfloor - 1 .\qedhere
  \]
\end{proof}

\begin{proof}[Proof of Theorem~\ref{thm:edge-separator}]
  Let \(G\) be a \(K_t\)-minor-free graph 
  of maximum degree \(\Delta\) with
  \(n\) vertices and \(m\) edges, and
  let \(w \colon V(G) \to [0, \frac{1}{2}]\)
  be a weight function such that
  \(\sum_{x \in V(G)}w(x) = 1\).

  By Theorem~\ref{thm:line-graph-tw}, we have
  \(\tw(L(G)) \le (t-1)\lfloor\sqrt{(t-3)\Delta m}+\Delta\rfloor - 1\).
  Let \((B_u)_{u \in V(T)}\) be a
  tree-decomposition of \(L(G)\) of minimum width. 
  For each \(x \in V(G)\), the set \(E(\{x\}, V(G))\)
  is a clique in \(L(G)\), so by
  Fact~\ref{fact}\ref{fact:complete-intersection}
  we can choose a node
  \(u(x) \in V(T)\) such that
  \(E(\{x\}, V(G))\subseteq B_{u(x)}\).
  
  For a subtree \(T'\) of \(T\), we define the
  weight \(w(T')\) of \(T'\) as the
  sum of weights \(w(x)\) of all
  vertices \(x \in V(G)\) such that \(u(x) \in V(T')\).
  Let us orient an edge \(u_1 u_2 \in E(T)\) from
  \(u_1\) to \(u_2\) when in \(T - u_1 u_2\)
  the component
  containing \(u_2\) has weight greater than \(\frac{1}{2}\)
  (and thus the component containing \(u_1\) has weight
  smaller than \(\frac{1}{2}\)).
  We do not orient an edge \(e \in E(T)\) in any direction if
  both components of \(T - e\) have weight exactly
  \(\frac{1}{2}\).

  Choose a node \(v\) in \(T\) such that
  no edge incident with \(v\) in \(T\)
  is oriented away from \(v\).
  We claim that \(F = B_v\) satisfies the theorem.
  Since the tree-decomposition has minimum width, 
  we have \(|F| \le (t-1)\lfloor\sqrt{(t-3)\Delta m}+\Delta\rfloor\). 
  Consider a component \(C\) of \(G - F\).  
  If $C$ consists of a single vertex, then \(\sum_{x \in V(C)}w(x) \le \frac{1}{2}\) clearly holds. 
  If $C$ has more than one vertex, then \(E(C) \neq \emptyset\), and \(L(G)[E(C)]\) is connected since \(C\) is connected.
  We have \(E(C) \cap F = \emptyset\), so by
  Fact~\ref{fact}\ref{fact:connected},
  there must exist a
  component \(T'\) of \(T - \{v\}\) such that every 
  node \(u \in V(T)\) with \(B_u \cap E(C) \neq \emptyset\) belongs to
  \(T'\). In particular, \(u(x) \in V(T')\) for each \(x \in V(C)\).
  By our choice of \(v\), the weight of \(T'\) is at most \(\frac{1}{2}\),
  so \(\sum_{x \in V(C)}w(x) \le \frac{1}{2}\).

  Kostochka~\cite{kostochka1982minimum} and Thomason~\cite{T1984} showed that
  a \(K_t\)-minor-free graph on \(n\) vertices has at most \(O(t\sqrt{\log t}\cdot n)\) edges,
  so \((t-1)\lfloor\sqrt{(t-3)\Delta m} + \Delta\rfloor \le  O\left(t^{2} (\log t)^{1/4}\cdot \sqrt{\Delta n}\right) \). This completes
  the proof.
\end{proof}

\section*{Acknowledgments}

Gwenaël Joret and Michał T.\ Seweryn thank David R.\ Wood for many helpful discussions, in particular regarding~\cite{illingworth2021alon}. 

\bibliographystyle{plain}
\bibliography{bibliography}

\end{document}